\documentclass[12pt,twoside]{amsart}
\usepackage[all]{xy}
\usepackage{graphicx}

\usepackage{longtable}

\title{Examples of singular toric varieties with 
certain numerical conditions} 
\author{Hiroshi Sato and Yusuke Suyama} 
\date{2019/12/17, version 0.12}
\address{Department of Applied Mathematics, Faculty of Sciences, 
Fukuoka University, 8-19-1, Nanakuma, Jonan-ku, Fukuoka 814-0180, Japan}
\email{hirosato@fukuoka-u.ac.jp}
\address{Department of Mathematics, Graduate School of 
Science, Osaka University, Toyonaka, Osaka 560-0043, Japan}
\email{y-suyama@cr.math.sci.osaka-u.ac.jp}

\makeatletter
    
    \@addtoreset{equation}{section}
\makeatother

\newcommand{\NE}[0]{{\operatorname{NE}}}
\newcommand{\N}[0]{{\operatorname{N}}}
\newcommand{\G}[0]{{\operatorname{G}}}
\newcommand{\Z}[0]{{\operatorname{Z}}}

\newtheorem{thm}{Theorem}[section]
\newtheorem{lem}[thm]{Lemma}

\newtheorem{prop}[thm]{Proposition}
\newtheorem*{claim}{Claim}

\theoremstyle{definition}
\newtheorem{ex}[thm]{Example}
\newtheorem{defn}[thm]{Definition}
\newtheorem{que}[thm]{Question}

\newtheorem*{ack}{Acknowledgments}       

\begin{document}
\bibliographystyle{amsalpha+}

\begin{abstract}
We give various examples of $\mathbb{Q}$-factorial 
projective toric varieties such that 
the sum of the squared torus invariant prime divisors 
is positive. We also determine the generators for 
the cone of effective $2$-cycles on a toric variety of 
Picard number two. 
This result is convenient to explain our examples. 
\end{abstract}

\maketitle

\tableofcontents
\section{Introduction} %%%%%%%%%%%%%%%%%%%

In \cite{satosumi}, the following concepts were introduced: 
\begin{defn}[{\cite[Definition 3.1]{satosumi}}]\label{gammapositive}
Let $X$ be a $\mathbb{Q}$-factorial projective toric 
$d$-fold. Put 
\[
\gamma_2=\gamma_2(X):=D_1^2+\cdots+D_n^2\in\N^2(X),
\]
where $D_1,\ldots,D_n$ be the torus invariant prime divisors. 

If $\gamma_2\cdot S>0$ (resp. $\ge 0$) for any subsurface 
$S\subset X$, 
then we say that $X$ is $\gamma_2$-{\em positive}  
(resp. $\gamma_2$-{\em nef}).   
\end{defn}

When $X$ is smooth, it is expected that $\gamma_2$-positive or 
$\gamma_2$-nef toric varieties have good geometric 
properties (see \cite{nobili}, \cite{sato1} and \cite{sato2}. 
Also see Questions \ref{question1} and \ref{question2} 
below). 
We should remark that $\frac{1}{2}\gamma_2(X)$ is the 
second Chern character ${\rm ch}_2(X)$ of $X$ in this case. 
It was confirmed that these properties hold for the case where 
$X$ is a $\mathbb{Q}$-factorial terminal 
toric Fano $3$-fold in \cite{satosumi}. 
Therefore, \cite{satosumi} posed the following questions:

\begin{que}[{\cite[Question 5.4]{satosumi}}]\label{question1}
Does there exist a $\mathbb{Q}$-factorial 
terminal projective $\gamma_2$-positive toric 
variety $X$ of $\rho(X)\ge 2$?
\end{que}

\begin{que}[{\cite[Question 5.6]{satosumi}}]\label{question2}
For any $\mathbb{Q}$-factorial terminal projective 
$\gamma_2$-nef toric 
$d$-fold of $\rho(X)\ge 2$, does one of the following 
hold? 
\begin{enumerate}
\item There exists a Fano contraction $\varphi:X\to \overline{X}$ 
such that $\overline{X}$ is a %$\mathbb{Q}$-factorial 
$\gamma_2$-nef toric $(d-1)$-fold. 
\item There exists a toric finite morphism 
$\pi:X'\to X$ such that $X'$ is a direct product of lower-dimensional  
%$\mathbb{Q}$-factorial 
$\gamma_2$-nef toric varieties.
\end{enumerate}
\end{que}

In this paper, we give answers for these questions 
by giving certain explicit examples (see Examples 
\ref{terminal4dim}, \ref{terminalanydim} and \ref{gorenstein3dim}, 
and Theorem \ref{Gorensteinsurface}). According to 
these examples, we see that 
higher-dimensional $\gamma_2$-positive or $\gamma_2$-nef 
singular toric varieties do {\em not} have good geometric properties 
like smooth cases. 

\begin{ack}
The first author was partially supported by JSPS KAKENHI 
Grant Number JP18K03262. 
The second author was partially supported 
by JSPS KAKENHI 
Grant Number JP18J00022. 
\end{ack}

%%%%%%%%%%%%%%%%%%%%%%%%%%%%%%%
\section{Preliminaries}

In this section, we introduce some basic results and notation of 
toric varieties. 
For the details, please see \cite{cls}, 
\cite{fulton} and  \cite{oda}. For the toric Mori theory, 
see also \cite{fujino-sato}, \cite[Chapter 14]{matsuki} 
and \cite{reid}.

Let $X=X_\Sigma$ be the toric $d$-fold associated to a fan 
$\Sigma$ in $N=\mathbb{Z}^d$ 
over an algebraically closed field $k$ of arbitrary characteristic. 
We will use the notation $\Sigma=\Sigma_X$ to denote the fan 
associated to a toric variety $X$. 
We denote the Picard number of $X$ by $\rho(X)$. 
Put $N_{\mathbb{R}}:=N\otimes\mathbb{R}$. 
There exists a one-to-one correspondence between 
the $r$-dimensional cones in $\Sigma$ and the torus invariant 
subvarieties of dimension $d-r$ in $X$. Let $\G(\Sigma)$ 
be the set of primitive generators for $1$-dimensional cones in $\Sigma$. 
Thus, for $v\in\G(\Sigma)$, 
we have the torus invariant prime divisor corresponding to 
$\mathbb{R}_{\ge 0}v\in\Sigma$.

\medskip

Let $X$ be a projective toric $d$-fold. 
For $1\le r\le d$, 
we put 
\[
\Z_r(X):=\{\mbox{the }r\mbox{-cycles on } X\},
\mbox{ while }
\Z^r(X):=\{\mbox{the }r\mbox{-cocycles on } X\}.
\]
We introduce the numerical equivalence $\equiv$ on $\Z_r(X)$ 
and $\Z^r(X)$ as follows: For $C\in\Z_r(X)$, we define 
$C\equiv 0$ if $D\cdot C=0$ for any $D\in \Z^r(X)$, while 
for $D\in\Z^r(X)$, we define 
$D\equiv 0$ if $D\cdot C=0$ for any $C\in \Z_r(X)$. 
We put 
\[
\N_r(X):=\left(\Z_r(X)\otimes{\mathbb R}\right)/\equiv, 
\mbox{ while }
\N^r(X):=\left(\Z^r(X)\otimes \mathbb R\right)/\equiv.
\] 
We denote the cone of effective 
$r$-cycles of $X$ by $\NE_r(X)\subset\N_r(X)$. 
$\NE_r(X)$ is a strongly convex rational polyhedral cone in $\N_r(X)$. 

For $\NE_1(X)=\NE(X)$, that is, the ordinary {\em Kleiman-Mori cone}, 
there is a good description of $1$-cycles. So, let $X$ be 
a $\mathbb{Q}$-factorial projective toric $d$-fold.  Let $C=C_\tau$ 
be the torus invariant curve corresponding to a $(d-1)$-dimensional cone 
$\tau$ generated by $x_1,\ldots,x_{d-1}$, 
where $x_1,\ldots,x_{d-1}\in\G(\Sigma)$. 
Then, there exist exactly two maximal cone $y_1+\tau$ and $y_2+\tau$ 
which contain $\tau$ as a face, where $y_1,y_2\in\G(\Sigma)$. 
So, we have the linear relation
\[
a_1y_1+a_2y_2+b_1x_1+\cdots+b_{d-1}x_{d-1}=0,
\]
where $a_1,a_2,b_1,\ldots,b_{d-1}\in\mathbb{Q}$ and $a_1,a_2>0$. 
We call this equality the {\em wall relation} for $\tau$. 
The wall relation is determined up to multiple of positive rational numbers. 
If $C$ spans an extremal ray of $\NE(X)$, 
we say that the wall relation for $\tau$ is {\em extremal}.

\medskip

%\section{Extremal rays of $\NE_2(X)$}\label{moricone} %%%%%

We end this section by determining the structure of $\NE_2(X)$, which 
is useful to describe the examples in Section \ref{counter}. 

\begin{thm}\label{ne2gene}
If $X=X_\Sigma$ is a $\mathbb{Q}$-factorial 
projective toric $d$-fold of $\rho(X)=2$, then 
$\NE_2(X)$ is generated by at most 
$3$ torus invariant surfaces. 
\end{thm}

\begin{proof}
First, we remark that \cite[Proposition 3.2]{nobili} says that $\NE_2(X)$ is 
generated by torus invariant surfaces. 

Reid's wall description of extremal rays of toric varieties 
tells us that there exist exactly two 
extremal wall relations
\[
a_1x_1+\cdots+a_mx_m=c_1y_1+\cdots+c_{n-1}y_{n-1},
\]
\[
b_1y_1+\cdots+b_ny_n=d_1x_1+\cdots+d_{m-1}x_{m-1}, 
\]
where $\G(\Sigma)=\{x_1,\ldots,x_m,y_1,\ldots,y_n\}$, 
$m,n\ge 2$, $m+n=d+2$, 
$a_1,\ldots,a_m,b_1,\ldots,b_n\in\mathbb{Q}_{>0}$, 
$c_1,\ldots,c_{n-1},d_1,\ldots,d_{m-1}\in\mathbb{Q}_{\ge 0}$.
Without loss of generality, we may assume that 
\[
0\le\frac{d_1}{a_1}\le\frac{d_2}{a_2}
\le\cdots\le\frac{d_{m-1}}{a_{m-1}}\mbox{ and }
0\le\frac{c_1}{b_1}\le\frac{c_2}{b_2}
\le\cdots\le\frac{c_{n-1}}{b_{n-1}}.
\]
By a $\mathbb{R}$-basis $\{x_1,\ldots,x_{m-1},
y_1,\ldots,y_{n-1}\}$ for $N_{\mathbb{R}}$, we obtain 
linear relations
\[
D_i-\frac{a_i}{a_m}D_m+\frac{d_i}{b_n}E_n=0\ (1\le i\le m-1),\ 
E_j-\frac{b_j}{b_n}E_n+\frac{c_j}{a_m}D_m=0\ (1\le j\le n-1) 
\]
in $\N^1(X)$, where $D_i$ and $E_j$ are the 
torus invariant prime divisors corresponding to $x_i$ and $y_j$, 
respectively. 
First, we show the following:
\begin{claim}%%%%%%%%%%%%%%%%%%%%%
For any $1\le i\le m-1$ and $1\le j\le n-1$, 
$D_m$ and $E_n$ are contained in 
the cone $\mathbb{R}_{\ge 0}D_i+
\mathbb{R}_{\ge 0}E_j\subset\N^1(X)$. 
\end{claim}
\begin{proof}[Proof of Claim]
If $d_i=0$, then we have 
$\frac{a_i}{a_m}D_m=D_i$. So, we may 
assume $d_i\neq 0$. 
By the above equalities, we have 
\[
\frac{b_j}{d_i}\left(
D_i-\frac{a_i}{a_m}D_m+\frac{d_i}{b_n}E_n
\right)+
E_j-\frac{b_j}{b_n}E_n+\frac{c_j}{a_m}D_m=0
\]
\[
\Longleftrightarrow\quad 
\frac{b_j}{d_i}D_i+E_j
=\left(
\frac{a_ib_j}{a_md_i}-\frac{c_j}{a_m}
\right)D_m, 
\]
where $\frac{a_ib_j}{a_md_i}-\frac{c_j}{a_m}$ 
has to be positive since $X$ is complete. 
The proof for $E_n$ is completely similar. 
\end{proof}%%%%%%%%%%%%%%%%%%%%%%%%
For $1\le i_1< i_2\le m-1$ and $1\le j_1< j_2\le n-1$, 
we have 
\[
\frac{a_m}{a_{i_1}}D_{i_1}=\frac{a_m}{a_{i_2}}D_{i_2}
+\frac{a_m}{b_n}\left(\frac{d_{i_2}}{a_{i_2}}-\frac{d_{i_1}}{a_{i_1}}\right)E_n
\mbox{ and }
\frac{b_n}{b_{j_1}}E_{j_1}=\frac{b_n}{b_{j_2}}E_{j_2}
+\frac{b_n}{a_m}\left(\frac{c_{j_2}}{b_{j_2}}-\frac{c_{j_1}}{b_{j_1}}\right)D_m.
\]
These equalities mean that $D_{i_1}\in\mathbb{R}_{\ge 0}D_{i_2}+
\mathbb{R}_{\ge 0}E_n\subset \N^1(X)$, while 
$E_{j_1}\in\mathbb{R}_{\ge 0}E_{j_2}+
\mathbb{R}_{\ge 0}D_m\subset \N^1(X)$. Therefore, 
any $2$-cycle $D_{i_1}\cdots D_{i_k}\cdot E_{j_1}\cdots E_{j_l}\ 
(k<m,\ l<n,\ k+l=d-2)$ is contained in the cone generated by 
\[
D_{p}\cdots D_{m-1}\cdot E_{q}\cdots E_{n-1}\ 
(p\ge 1,\ q\ge 1,\ p+q=4)
\]
in $\NE_2(X)$. One can easily see that the possibilities for 
$(p,q)$ are $(1,3)$, $(2,2)$ and $(3,1)$. Thus, 
$\NE_2(X)$ is generated by the three $2$-cycles
\[
S_1:=D_1\cdots D_{m-1}\cdot E_{3}\cdots E_{n-1},\ 
S_2:=D_2\cdots D_{m-1}\cdot E_2\cdots E_{n-1}, 
\]
\[ 
\mbox{and }S_3:=D_3\cdots D_{m-1}\cdot E_1\cdots E_{n-1}, 
\]
where $S_1=0$ (resp. $S_3=0$) if 
$n=2$ (resp. $m=2$). 
These $2$-cycles are obtained by multiplying 
some torus invariant surfaces by positive rational numbers.
\end{proof}

By Theorem \ref{ne2gene}, 
in order to prove the positivity (resp. non-negativity) of $\gamma_2(X)$, 
it is sufficient to check 
the positivity (resp. non-negativity) for the above three $2$-cycles. 
Furthermore, 
\cite[Proposition 3.4]{satosumi} says that $\gamma_2(X)\cdot S_1>0$ 
and $\gamma_2(X)\cdot S_3>0$. 
So, only we have to do is to check 
the positivity (resp. non-negativity) for $S_2$. 
We remark that $\rho(S_2)=2$. 
So, we can apply \cite[Proposition 3.5]{satosumi}. 
We describe them here for the reader's convenience: 
Let $X=X_\Sigma$ be a $\mathbb{Q}$-factorial projective toric 
$d$-fold, and $S\subset X$ a torus invariant subsurface 
of $\rho(S)=2$. 
Let $\tau\in\Sigma$ be a $(d-2)$-dimensional cone 
associated to $S$ and 
$\tau\cap\G(\Sigma)=\{x_1,\ldots,x_{d-2}\}$. 
There exist exactly $4$ maximal cones 
\[
\mathbb{R}_{\ge 0}y_1+\mathbb{R}_{\ge 0}y_3+\tau,\ 
\mathbb{R}_{\ge 0}y_2+\mathbb{R}_{\ge 0}y_3+\tau,\ 
\mathbb{R}_{\ge 0}y_1+\mathbb{R}_{\ge 0}y_4+\tau,\ 
\mathbb{R}_{\ge 0}y_2+\mathbb{R}_{\ge 0}y_4+\tau 
\] 
in $\Sigma$, where $\{y_1,y_2,y_3,y_4\}\subset\G(\Sigma)$. 
Let 
\[
b_1y_1+b_2y_2+c_3y_3+a_1x_1+\cdots+a_{d-2}x_{d-2}=0
\mbox{ and }
\]
\[
b_3y_3+b_4y_4+c_1y_1+e_1x_1+\cdots+e_{d-2}x_{d-2}=0
\]
be the wall relations corresponding to 
$(d-1)$-dimensional cones 
$\mathbb{R}_{\ge 0}y_3+\tau$ and  
$\mathbb{R}_{\ge 0}y_1+\tau$, 
respectively, where $a_1,\ldots,a_{d-2},b_1,b_2,b_3,b_4,
c_1,c_3,e_1,\ldots,e_{d-2}\in\mathbb{Q}$ and 
$b_1,b_2,b_3,b_4>0$. 
Then, the following holds:
\begin{prop}[{\cite[Proposition 3.4]{satosumi}}]\label{rho2lemma}
There exists a positive rational number $\alpha$ such that 
\[
\alpha \gamma_2(X)\cdot S=
-b_3c_1\left(b_1^2+b_2^2+c_3^2+a_1^2+\cdots+a_{d-2}^2\right)
\]
\[
+2b_1b_3\left(b_1c_1+b_3c_3+a_1e_1+\cdots+a_{d-2}e_{d-2}\right)
-b_1c_3\left(b_3^2+b_4^2+c_1^2+e_1^2+\cdots+e_{d-2}^2\right).
\]
\end{prop}

\section{Examples of $\gamma_2$-positive toric varieties} %%%%%%%%%%
\label{counter} %%%%%%%%%%%%%%
%%%%%%%%%%%%%%%%%%%%%%%%%%%%%%%%%%%%%
%%%%%%%%%%%%%%%%%%%%%%%%%

We need the following lemma to explain the singularities 
in the examples below. 

\begin{lem}\label{terminalsing}
Let $d\ge 3$ and $e_1,\ldots,e_d$ the standard basis for $N$. 
Put 
\[
x_1:=e_1,\ldots,x_{d-1}:=e_{d-1},x_d:=ce_d-\sum_{i=p}^{d-1}e_i,
\]
where $1\le p\le d-1$, $c\in\mathbb{Z}$ and $0<c<d-p+1$. 
Then, the cone $\mathbb{R}_{\ge 0}x_1+\cdots+\mathbb{R}_{\ge 0}x_d
\subset N_{\mathbb{R}}$ 
is terminal. 
\end{lem}

\begin{proof}
The hyperplane passing through $x_1,\ldots,x_d$ is 
\[
\left\{(t_1,\ldots,t_d)\in N_{\mathbb{R}}^d\,\left|\,
t_1+\cdots+t_{d-1}+\frac{d-p+1}{c}t_d=1\right.\right\}.
\]
For $(a_1,\ldots,a_d)\in\mathbb{Q}_{\ge 0}^d$, 
suppose that 
\[
x:=a_1x_1+\cdots+a_dx_d=
a_1e_1+\cdots+a_{p-1}e_{p-1}+
(a_p-a_d)e_p+\cdots+
(a_{d-1}-a_d)e_{d-1}+ca_de_d
\in\mathbb{Z}^d
\]
and that 
\[
a_1+\cdots+a_{p-1}+
(a_p-a_d)+\cdots+(a_{d-1}-a_d)+
\frac{d-p+1}{c}\times ca_d
=a_1+\cdots+a_d\le 1.
\]
If $a_i=1$ for $1\le i\le d$, then $x=x_i$. So we may assume 
$a_1,\ldots,a_d<1$. 
Then, since $a_1,\ldots,a_{p-1}\in\mathbb{Z}$, 
$a_1=\cdots=a_{p-1}=0$. So, we have 
$0\le a_p+\cdots+a_d\le 1$. 
For any $p\le i\le d-1$, we have $-1< a_i-a_d< 1$. However, 
$a_i-a_d\in\mathbb{Z}$ means that $a_i-a_d=0$. 
If $a_d\neq 0$, then $ca_d\ge 1$ holds 
because $ca_d\in\mathbb{Z}$. This is 
impossible, since 
\[
a_p+\cdots+a_d=(d-p+1)\times a_d\ge
\frac{d-p+1}{c}>1.
\]
Therefore, $a_p=\cdots=a_d=0$. Thus, $x\in \{x_1,\ldots,x_d,0\}$. 
\end{proof}

The following is an answer to 
Question \ref{question1}. Moreover, this is 
a counterexample to Question \ref{question2}, too. 

\begin{ex}\label{terminal4dim}
Let $X=X_{\Sigma}$ be a $\mathbb{Q}$-factorial 
terminal  
toric Fano $4$-fold such that 
the primitive generators of $1$-dimensional 
cones in $\Sigma$ are 
\[
x_1=(1,0,0,0),\ x_2=(0,1,0,0),\ x_3=(0,0,1,0),
\]
\[
x_4=(0,0,0,1), x_5=(-1,-2,-1,0), 
x_6=(0,-1,-2,-1).
\]
The singular locus of $X$ is $S_{1,5}\cup S_{4,6}$, where 
$S_{1,5}$ and $S_{4,6}$ are the torus invariant surfaces 
corresponding to $\mathbb{R}_{\ge 0}x_1+\mathbb{R}_{\ge 0}x_5$ 
and $\mathbb{R}_{\ge 0}x_4+\mathbb{R}_{\ge 0}x_6$, 
respectively. 
One can easily see that $X$ is terminal by Lemma \ref{terminalsing}. 
The extremal wall relations of $\Sigma$ are 
\[
2x_1+3x_2+2x_5=x_4+x_6\mbox{ and }
3x_3+2x_4+2x_6=x_1+x_5.
\]
Let $D_1,\ldots,D_6$ be the torus invariant 
prime divisors corresponding to 
$x_1,\ldots,x_6$, respectively. 
Theorem \ref{ne2gene} tells us that 
it is sufficient to show the positivity 
for $D_5D_6$. 
The wall relations associated to 
$\mathbb{R}_{\ge 0}x_1+\mathbb{R}_{\ge 0}x_5+\mathbb{R}_{\ge 0}x_6$ 
and $\mathbb{R}_{\ge 0}x_3+\mathbb{R}_{\ge 0}x_5+\mathbb{R}_{\ge 0}x_6$ 
are 
\[
3x_3+2x_4-x_1-x_5+2x_6=0\mbox{ and }
x_1+2x_2+x_3+x_5=0, 
\]
respectively. By Proposition \ref{rho2lemma}, 
there exists a positive rational number $\alpha$ such that
\[
\alpha\gamma_2(X)\cdot D_5D_6  =  -1\times 1\times (3^2+2^2+(-1)^2+(-1)^2+2^2)
+2\times 3\times 1\times (3\times 1+1\times (-1)+(-1)\times 1)
\]
\[
-3\times(-1)\times(1^2+2^2+1^2+1^2) 
 =  8>0.
\]
Therefore, $X$ is $\gamma_2$-positive, 
but $\rho(X)=2$. We should remark that 
$\G(\Sigma)$ has no centrally symmetric pair. 
\end{ex}

For any dimension $d\ge 4$, 
there exists a toric $d$-fold satisfying the condition of 
Question \ref{question1}: 
\begin{ex}\label{terminalanydim}
Let $d\ge 4$ and $\{e_1,\ldots,e_d\}$ the standard basis for $N=\mathbb{Z}^d$. 
Put 
\[
x_1:=e_1,\ \ldots,\ x_{d-2}:=e_{d-2},\ 
x_{d-1}:=-\left(e_1+\cdots+e_{d-2}+(d-2)e_{d-1}\right),\ 
x_d:=e_{d-1},\]
\[
y_1:=-(e_{d-1}+e_d),\ y_2=e_d.
\]
Let $X=X_{\Sigma}$ be the $\mathbb{Q}$-factorial terminal 
toric Fano $d$-fold of $\rho(X)=2$ 
such that $\G(\Sigma)=\{x_1,\ldots,x_{d},y_1,y_2\}$. 
The singular locus of $X$ is the torus invariant curve corresponding to 
the cone $\mathbb{R}_{\ge 0}x_1+\cdots+\mathbb{R}_{\ge 0}x_{d-1}$. 
One can easily confirm that this singularity is terminal by Lemma 
\ref{terminalsing}. 
The extremal wall relations of $\Sigma$ are 
\[
x_1+\cdots+x_{d-1}+(d-2)x_d=0\mbox{ and }
(d-2)y_1+(d-2)y_2=x_1+\cdots+x_{d-1}.
\]
By Theorem \ref{ne2gene}, all we have to do is to show 
$\gamma_2(X)\cdot D_2\cdots D_{d-1}>0$, where 
$D_1,\dots,D_d,E_1,E_2$ are the torus invariant prime divisors 
corresponding to $x_1,\ldots,x_d,y_1,y_2$, respectively. 
The wall relations associated to 
\[
\mathbb{R}_{\ge 0}x_1+\mathbb{R}_{\ge 0}x_2+
\cdots+\mathbb{R}_{\ge 0}x_{d-1}\mbox{ and }
\mathbb{R}_{\ge 0}y_1+\mathbb{R}_{\ge 0}x_2+
\cdots+\mathbb{R}_{\ge 0}x_{d-1}
\]
are 
\[
(d-2)y_1+(d-2)y_2-x_1-x_2-\cdots -x_{d-1}=0\mbox{ and }
x_1+(d-2)x_d+x_2+\cdots+x_{d-1}=0,
\]
respectively. Proposition \ref{rho2lemma} says that 
for $\alpha\in\mathbb{Q}_{>0}$, we have 
\[
\alpha\gamma_2(X)\cdot D_2\cdots D_{d-1}
\]
\[
=2\times (d-2)\times 1\times
(-1)\times (d-1)-(d-2)\times(-1)\times
\left(1^2+(d-2)^2+1^2\times (d-2)\right)
\]
\[
=(d-2)^3-(d-2)(d-1)=(d-2)((d-3)^2+(d-4))>0.
\]
Thus, $X$ is $\gamma_2$-positive. 
Moreover, $\G(\Sigma)$ has no centrally symmetric pair 
in this case, too. 
\end{ex}

%\bigskip

Next, we consider Question \ref{question1} 
for {\em Gorenstein} $\mathbb{Q}$-factorial 
projective 
toric $d$-folds. We remark that 
there exists a counterexample to 
Question \ref{question2} in this 
situation 
(see \cite[Remark 5.7]{satosumi}). 

The following is the answer to Question \ref{question1} 
for $d=2$.

\begin{thm}\label{Gorensteinsurface}
Let $S$ be a Gorenstein projective 
toric surface. Then, $S$ is 
$\gamma_2$-positive if and only if 
$\rho(S)=1$. 
\end{thm}
\begin{proof}
If $S$ is nonsingular, then the statement is obviously true 
(for example, see \cite[Proposition 4.3]{sato2}). 

Suppose $\rho(S)\ge 2$. 
Only we have to do is to show that 
$S$ is not $\gamma_2$-positive.

First, we remark that for a blow-up 
$\psi: S_1\to S_2$ between 
smooth projective toric surfaces 
$S_1$ and $S_2$, we have 
$\gamma_2(S_2)-\gamma_2(S_1)=3$. 

Next, we investigate primitive 
crepant contractions. So, let 
$\psi: S_1\to S_2$ be a toric 
morphism between Gorenstein projective 
toric surfaces $S_1$ and $S_2$ such that 
$\G(\Sigma_{S_1})=
\G(\Sigma_{S_2})\cup \{y\}$ 
and $ax_1+bx_2=qy$ for some 
$2$-dimensional cone $\mathbb{R}_{\ge 0}x_1+
\mathbb{R}_{\ge 0}x_2\in\Sigma_{S_2}$, 
where $a,b,q$ are coprime positive 
integers and $x_1,x_2\in\G(\Sigma_{S_2})$. 
Then, \cite[Proposition 4.2]{satosumi} 
says that 
\[
\gamma_2(S_1)=\gamma_2(S_2)
-\frac{a^2+b^2+q^2}{abq}.
\]
Since $\psi$ is crepant if and only if 
$a+b=q$, this equality is equivalent to 
\[
\gamma_2(S_2)-\gamma_2(S_1)=
\frac{a^2+b^2+(a+b)^2}{ab(a+b)}=
2\left(\frac{1}{a}+\frac{1}{b}-
\frac{1}{a+b}\right).
\]
Put 
\[
f(a,b):= \left(\frac{1}{a}+\frac{1}{b}-
\frac{1}{a+b}\right).
\]
Then, 
\begin{eqnarray*}
f(a+1,b)-f(a,b) & = & \left(\frac{1}{a+1}-
\frac{1}{a}\right)-\left(\frac{1}{a+b+1}-
\frac{1}{a+b}\right) \\
& = &
-\frac{1}{a(a+1)}+\frac{1}{(a+b)(a+b+1)}<0.
\end{eqnarray*}
This means that $f(a,b)$ takes the 
maximum value at $(a,b)=(1,1)$. Thus, 
we have 
\[
\gamma_2(S_2)-\gamma_2(S_1)\le 
\frac{1^2+1^2+2^2}{1\times 1\times 2}=3.
\]

There exists the crepant resolution 
$\pi:\overline{S}\to S$ which is a 
finite succession of primitive 
crepant contractions as above. 
On the other hand, 
there exists a toric morphism 
$\varphi:\overline{S}\to S'$ 
which is a finite succession of 
blow-ups such that 
$S'$ is a smooth projective toric 
surface of $\rho(S')=\rho(S)$. 
Thus, we have 
\[
\gamma_2(S')-\gamma_2(\overline{S})
=3\left(\rho(\overline{S})-\rho(S')\right), 
\]
while 
\[
\gamma_2(S)-\gamma_2(\overline{S})
\le 3\left(\rho(\overline{S})-\rho(S)\right). 
\]
Therefore, 
$\gamma_2(S)\le\gamma_2(S')\le 0$, that is, 
$S$ is not $\gamma_2$-positive.
\end{proof}

However, there exists a 
Gorenstein $\mathbb{Q}$-factorial 
projective $\gamma_2$-positive 
toric $3$-fold $X$ of $\rho(X)=2$: 

\begin{ex}\label{gorenstein3dim}
Let $X=X_{\Sigma}$ be a $\mathbb{Q}$-factorial 
Gorenstein 
toric Fano $3$-fold such that 
the primitive generators of $1$-dimensional 
cones in $\Sigma$ are
\[
x_1=(1,0,0),\ x_2=(0,1,0),\ x_3=(0,0,1),\  x_4=(0,-2,-1),\ x_5=(-1,-1,0).
\]
The singular locus of $X$ is the torus invariant curve corresponding to 
the cone $\mathbb{R}_{\ge 0}x_3+\mathbb{R}_{\ge 0}x_4$. 
The hyperplane passing through $x_1,x_3,x_4$ and 
$x_3,x_4,x_5$ are 
\[
\left\{(t_1,t_2,t_3)\in N_{\mathbb{R}}^3\,\left|\,
t_1-t_2+t_3=1\right.\right\}\mbox{ and }
\left\{(t_1,t_2,t_3)\in N_{\mathbb{R}}^3\,\left|\,
-t_2+t_3=1\right.\right\},
\]
respectively. 
Thus, $X$ is Gorenstein. 
There exist exactly two extremal 
wall relations 
\[
2x_1+2x_5=x_3+x_4\mbox{ and }
2x_2+x_3+x_4=0.
\]
Let $D_1,\ldots,D_5$ be the torus invariant 
prime divisors corresponding to 
$x_1,\ldots,x_5$, respectively. 
By Theorem \ref{ne2gene}, it is sufficient to 
check the positivity for $D_4$. 
The wall relations associated to 
$\mathbb{R}_{\ge 0}x_1+\mathbb{R}_{\ge 0}x_4$ 
and $\mathbb{R}_{\ge 0}x_2+\mathbb{R}_{\ge 0}x_4$ 
are 
\[
2x_2+x_3+x_4=0\mbox{ and }
x_1+x_5+x_2=0, 
\]
respectively. By Proposition \ref{rho2lemma}, 
there exists a positive rational number $\alpha$ such that
\[
\alpha\gamma_2(X)\cdot D_4  =  -1\times 1\times (2^2+1^2+1^2)
+2\times 2\times 1\times (2\times 1)
-2\times 0\times(1^2+1^2+1^2) 
 =  2>0.
\]
Therefore, $X$ is $\gamma_2$-positive, but 
$\rho(X)=2$. 
\end{ex}

\end{document}